\documentclass[a4, 12pt]{amsart}
\usepackage{amssymb,latexsym}
\usepackage{color}
\usepackage{graphicx}
\usepackage[left=25mm, right=25mm, top=2mm, bottom=5mm, includefoot, includehead]{geometry}
\theoremstyle{plain}
\newtheorem{theorem}{Theorem}
\newtheorem{corollary}[theorem]{Corollary} 
\newtheorem{proposition}[theorem]{Proposition}

\newtheorem*{definition}{Definition}

\theoremstyle{definition}
\newtheorem{remark}[theorem]{Remark}

\newcommand{\enm}[1]{\ensuremath{#1}}          %

\newcommand{\cal}[1]{\mathcal{#1}}

\newcommand{\PP}{\enm{\mathbb{P}}}

\newcommand{\Ii}{\enm{\cal{I}}}

\newcommand{\Oo}{\enm{\cal{O}}}

\newcommand{\TT}{\mathbb{T}}

\renewcommand{\phi}{\varphi}
\renewcommand{\theta}{\vartheta}
\renewcommand{\epsilon}{\varepsilon}

\DeclareMathOperator{\reg}{reg}

\renewcommand{\to}[1][]{\xrightarrow{\ #1\ }}

\newcommand{\old}[1]{}

\makeatletter
\@namedef{subjclassname@2020}{%
	\textup{2020} Mathematics Subject Classification}
\makeatother

\date{}


\title{A note on very ample Terracini loci}

\author{Edoardo Ballico}
\address{Universit\`{a} di Trento, Via Sommarive 14,  38123 Povo (Trento), Italy}
\email{edoardo.ballico@unitn.it}

\author{Emanuele Ventura}
\address{Politecnico di Torino, Dipartimento di Scienze Matematiche ``G. L. Lagrange'', Corso Duca degli Abruzzi 24\\
10129 Torino, Italy}
\email{emanuele.ventura@polito.it}

\begin{document}

\maketitle

\begin{abstract}
In this short note we show that, for any ample embedding of a variety of dimension at least two in a projective space, all high enough degree Veronese re-embeddings have non-empty Terracini loci.  
\end{abstract}

\section{Introduction}
Terracini loci were introduced by the first author and Chiantini in \cite{BC}. Their emptiness implies non-defectivity of secant varieties due to the celebrated Terracini's lemma, whereas the converse is not true: there exist non-empty Terracini loci even in the presence of non-defective secants. This triggered the interest for this geometric notion, leading to the results in the aforementioned article. The Terracini locus has been the subject of recent investigations \cite{BBS, CG}, especially for Segre and Veronese varieties, that are crucial in the context of tensors. We start off by defining set-theoretically these loci. 

\begin{definition}
Let $X\subset \PP^N$ be a non-degenerate projective variety of dimension $n\geq 1$ over an algebraically closed field $\mathbb K$. Let $S\subset X_{\mathrm{reg}}$ be a finite subset of smooth points of $X$ whose cardinality is $k$. Let $(2S,X)$ be the union of the corresponding $2$-fat points $(2p,X)$ supported at the points $p\in S$. Then $S$ is in the $k$th {\it Terracini locus} $\TT_k(X)$ if and only if $h^0(\Ii _{(2S,X)}(1))>0$ and $h^1(\Ii_{(2S,X)}(1)) >0$. Equivalently, $S$ is in $\TT_k(X)$ whenever the $n$-dimensional tangent spaces $T_pX$, for $p\in S$, are linearly dependent and their projective linear span is not the ambient space $\PP^N$.
\end{definition}

A consequence of a deep result of Alexander and Hirschowitz \cite[Theorem 1.1 and Corollary 1.2]{AH} (where in their notation one chooses $m=2$) states that for any projective variety $X$ there exists a very ample embedding such that all the secant varieties of $X$ under this embedding are non-defective.  The aim of this note is to point out that, even in this very ample regime, the emptiness of the corresponding Terracini locus {\it does not} generally hold. Thus we answer 
in the negative the question whether a statement similar to the one by Alexander and Hirschowitz works for Terracini loci. \\

\section{Very ample regime}

Let  $\mathbb K$ be an algebraically closed field and let $X$ be a projective variety of dimension $n$ over $\mathbb K$. We say that an embeddeding $X\subset \PP^r$ of $X$ is not secant defective if for each positive integer $k$ the $k$-secant variety of $X$ has dimension $\min \{r,k(n+1)-1\}$. For a very ample line bundle $L$ on $X$, let $\nu_L: X\to |L|^\vee$ denote the associated embedding. The $k$th secant variety and the $k$th Terracini locus of $\nu_L(X)$ are denoted $\sigma _k(\nu_{L}(X))$ and $\TT_k(\nu_{L}(X))$, respectively.  \\
We say that $\nu_{L}(X)$ is {\it secant non-defective} if $\sigma _k(\nu_{L}(X))$ is non-defective for {\it every} $k\geq 1$. 

\begin{theorem}\label{main}
Let $n\geq 2$ and $X$ be as above. Let $F, L\in \mathrm{Pic}(X)$, where $L$ is an ample line bundle. Then there exists an integer $m_0$ \textnormal{(}depending only on $X$, $F$, $L$\textnormal{)} such that for all $m\ge m_0$ the line bundle $F+mL$ is very ample, $\nu_{F+mL}(X)$ is secant non-defective, and there exists $k>0$ such that $\sigma _k(\nu_{F+mL}(X))\ne |F+mL|^\vee$ and $\TT_k(\nu_{F+mL}(X)) \ne \emptyset$.
\end{theorem}

\begin{proof}
Let $L = \mathcal L(D)$ and define $\alpha = D\cdots D>0$, the $n$ times self-intersection of the Cartier divisor $D$. Fix an integral curve $Y\subset X$ such that $Y\cap X_{\reg}\ne \emptyset$, where $Y$ is possibly singular. Let $\beta = Y\cdot D\cdots D$, the intersection of $Y$ with $n-1$ copies of $D$, i.e. $\beta = \deg (L_{|Y})$ and $\beta>0$ because $L$ is ample. Fix a real number $\epsilon$ such that $\alpha > \epsilon >0$. By the result of  Alexander and Hirschowitz \cite[Theorem 1.1]{AH}, by the asymptotic Riemann-Roch and by the ampleness of $L$, we find  an integer $m_1$ such that for all $m\ge m_1$ we have that: $F+mL$ is very ample, $\nu_{F+mL}(X)$ is secant non-defective, and $h^0(F+mL) \ge \frac{\alpha-\epsilon}{n!}m^n$. 

Thus, for $1\le k <\left\lfloor \frac{\alpha  -\epsilon}{(n+1)!}m^n\right\rfloor$, we have $\sigma _k(\nu_{F+mL}(X))\subsetneq |F+mL|^\vee$. By the asymptotic Riemann-Roch, $h^0(Y, (F+mL){|_Y})$ grows like a linear function of the form $\beta m$. Therefore there exists $m_0\geq m_1$ such that for all $m\geq m_0$ one has $1\leq h^0(Y, (F+mL){|_Y})/2 < \left\lfloor \frac{\alpha-\epsilon}{(n+1)!}m^n\right\rfloor$. 

Define $k-1= \lceil h^0(Y, (F+mL){|_Y})/2\rceil$. Note that the projective linear span of the curve $Y$ has dimension $\dim \langle Y\rangle \leq 2k-3$. Fix a set $S\subset Y\cap X_{\reg}$ with cardinality $k$. The zero-dimensional scheme $(2S,X)\cap Y\subset Y$ has degree at least $2k$. Hence, if $(2S,X)\cap Y\subset Y$ was linearly independent, then its projective linear span would be at least $(2k-1)$-dimensional. Therefore $(2S,X)\cap Y$ is linearly dependent, i.e. $h^1(\Ii_{(2S,X)\cap Y}(1))>0$. Moreover, since $k <\left\lfloor \frac{\alpha  -\epsilon}{(n+1)!}m^n\right\rfloor$ and $\deg ((2S,X))  = k(n+1)$, the projective linear span of this scheme cannot fill the ambient space, i.e. one has $h^0(\Ii_{(2S,X)}(1)) >0$.

Now, let $Z\subset W$ be two zero-dimensional schemes. Then one has the exact sequence of sheaves 
\[
0\longrightarrow \Ii_W(1)\longrightarrow \Ii_Z(1) \longrightarrow \Ii_Z(1)/\Ii_W(1) \longrightarrow 0. 
\]
Here the cokernel sheaf is either zero or supported on a zero-dimensional scheme. Taking the long exact sequence in cohomology, we then find a surjective 
map in cohomology $H^1(\Ii_W(1))\twoheadrightarrow H^1(\Ii_Z(1))$. The zero-dimensional scheme $(2S,X)\cap Y$ is a closed subscheme of $(2S,X)$ and so we likewise have 
a surjection
\[
H^1(\Ii_{(2S,X)}(1))\twoheadrightarrow H^1(\Ii_{(2S,X)\cap Y}(1)). 
\]
Therefore $h^1(\Ii_{(2S,X)}(1))>0$ too. So any collection of $k$ smooth points of $Y\cap X_{\reg}$ is in the $k$th Terracini locus of $\nu_{F+mL}(X)$.
\end{proof}

\begin{remark}
Let $X\subset \PP^N$ be a projective variety with $\dim X=n\geq 2$ and consider $\nu_d(X)$. For any integer $k>0$, the set $S^k\nu_d(X_{\mathrm{reg}})$ of all subsets of $\nu_d(X_{\mathrm{reg}})$  with cardinality $k$ is a variety of dimension $kn$. For $d\gg 0$, the families of $S\in \TT_k(\nu_d(X))$ we found in the proof of Theorem \ref{main} on a fixed curve $Y$ have codimension $k$ in $S^k\nu_d(X_{\mathrm{reg}})$. Varying $Y$, we do not decrease significantly the codimension of $\TT_k(\nu_d(X))$ in $S^k\nu_d(X_{\mathrm{reg}})$: the magnitude of this is $O(k)$. We do not have examples for which, when $k$ is increasing with $d$, $\TT_k(\nu_d(X))$ has codimension $1$ in  $S^k\nu_d(X_{\mathrm{reg}})$, which is the least codimension allowed in view of the secant non-defectivity result in \cite{AH}. 
\end{remark}

\begin{proposition}\label{uu1}
Let $N\geq 1$ and let $C\subset \PP^N$ be a smooth and non-degenerate rational curve of degree $d$. For all $d'\ge d+1-N$, the curve $\nu_{d'}(C)\subset \langle \nu_{d'}(C)\rangle$ has empty Terracini loci.
\begin{proof}
Suppose $N=d=1$ so that $C=\PP^1$. For $d'\geq 1$, consider the rational normal curve $\nu_{d'}(\PP^1)$. Its $k$th Terracini locus 
consists of those subsets $S\subset \PP^1$ such that $(2S, \nu_{d'}(\PP^1))$ does not span $\langle \nu_{d'}(C)\rangle$, i.e. $h^0(\Ii _{2S}(d')) >0$, and such that $h^1(\Ii_{2S}(d'))> 0$. Since $C=\PP^1$,
for any zero-dimensional scheme $Z\subset C$ either $h^0(\Ii_Z(d')) =0$ or $h^1(\Ii_Z(d')) =0$.
Hence any Terracini locus of the rational normal curve $\nu_{d'}(C)$ is empty. 

For the general case, let $d\geq 2$ and $d'\ge d+1-N$. One has $h^1(\Ii_C(d')) =0$ \cite[Theorem p. 492]{GLP}. Hence $\nu_{d'}(C)$ is an embedding of $\PP^1$ by the complete linear system $|\Oo_{\PP^1}(d\cdot d')|$. So this has empty Terracini loci by the first part. 
\end{proof}
\end{proposition}

The case of curves with positive arithmetic genus is treated in the following proposition. Here different behaviours appear according to the parity of the degree. 

\begin{proposition}\label{arithgenus>0}
Let $C$ be an integral projective curve over $\mathbb K$, with $\mathrm{char}(\mathbb K)\neq 2$, whose arithmetic genus is $g>0$. Let $F$ and $L$ be line bundles on $C$, where $L$ is ample, of degrees $\alpha = \deg (L)$ and $\beta = \deg (F)$. For each integer $m>0$, consider the complete linear system $|F +mL|$. Assume that $\beta +m\alpha \geq 4g+2$  and assume that $\beta +m\alpha$ is even. Then 
$\nu_{F+mL}(C)$ has a non-empty Terracini locus.
\end{proposition}

\begin{proof}
Recall that a line bundle  $E$ on $C$ is very ample if $\deg(E)\geq 2g+1$ \cite[Corollary 3.2, Chapter IV]{Hart}. Since the Picard group $\mathrm{Pic}^0(C)$ is a quasi-projective irreducible group and $\mathrm{char}(\mathbb K)\neq 2$, the kernel of the multiplication morphism $\otimes 2: \mathrm{Pic}^0(C)\to \mathrm{Pic}^0(C)$ is finite. So this morphism is surjective. Since $\deg (F+mL)$ is even and $\otimes 2$ is surjective, there is a line bundle $R_m$ such that
$R_m^{\otimes 2} \cong F+mL$. Thus $\deg (R_m) = (\beta+m\alpha)/2$.
Since $\beta +m\alpha \ge 4g+2$, the line bundle $R_m$ is very ample. Thus $|R_m|\ne \emptyset$ and a general $S\in |R_m|$ consists of $k$ distinct reduced points and $S\subset C_{\reg}$. Note that $2S\in |F+mL|$ and hence $\langle 2\nu_{F+mL}(S)\rangle \subsetneq |F+mL|^{\vee}$ is a hyperplane. Since $\deg (F+mL) >2g-1$, one has $h^0(F+mL) = \deg (F+mL) +1-g=\deg(2S)+1-g$. Since $g>0$, $2S$ does not give $\deg (2S)$ independent conditions to $|F+mL|$. Then, by definition, $\nu_{F+mL}(S)$ is in the $k$th Terracini locus of $\nu_{F+mL}(C)$.
\end{proof}

\begin{corollary}\label{uu3}
Let $C\subset \PP^N$ be an integral and non-degenerate projective curve with arithmetic genus $g=1$ of degree $d$ over $\mathbb K$, with $\mathrm{char}(\mathbb K)\neq 2$. If $d'\geq d+1-N$ and $d\cdot d'$ is even, then $\TT_{d\cdot d'/2}(\nu_d(C))\ne \emptyset$. If $d\cdot d'$ is odd, then all Terracini loci of $\nu_{d'}(C)$ are empty. 
\begin{proof}
Since $d'\ge d+1-N$, we have $h^1(\Ii _C(d')) =0$ \cite[Theorem p. 492]{GLP}. Hence $\nu_{d'}(C)$ is an embedding of $C$ by a complete linear system. 
By Proposition \ref{arithgenus>0}, if $d\cdot d'$ is even, then $\TT_{d\cdot d'/2}(\nu_{d'}(C))\ne \emptyset$. 

Suppose a line bundle $L$ on $C$ has $\deg(L)=2m+1$; let $S\subset C_{\mathrm{reg}}$ have cardinality $k$. Then $\deg(L(-2S)) = 2(m-k)+1\neq 0$. If $\deg(L(-2S))<0$, then $h^0(L(-2S)) = 0$. If $\deg(L(-2S))>0$, by Serre duality, we find $h^1(L(-2S)) = 0$. Therefore any Terracini locus is empty. 
\end{proof}
\end{corollary}

\begin{small}
\noindent {\bf Acknowledgements.}
The open-ended problem of looking at very ample Terracini loci was proposed by Luca Chiantini and Ciro Ciliberto (along with many other interesting problems) during the {\it Geometry of Secants Workshop} held in October 2022, within the AGATES semester at University of Warsaw and IMPAN. We warmly thank Chiantini, Ciliberto, and the organizers of the semester and of the workshop. This work is supported by  the Thematic Research Programme ``Tensors: geometry, complexity and quantum entanglement'', University of Warsaw, Excellence Initiative -- Research University and the Simons Foundation Award No. 663281 granted to the Institute of Mathematics of the Polish Academy of Sciences for the years 2021-2023. \\
We thank an anonymous referee for useful comments and corrections. 
\end{small}

\end{document}